\documentclass[a4paper]{amsart}

\usepackage{amsmath, amssymb, amsthm}
\usepackage{xypic, graphicx}
\usepackage{enumerate}

\usepackage{aliascnt} 
\usepackage[colorlinks=true,linkcolor=blue,citecolor=blue,urlcolor=blue,citebordercolor={0 0 1},urlbordercolor={0 0 1},linkbordercolor={0 0 1}]{hyperref} 
\usepackage[nameinlink]{cleveref}

\theoremstyle{plain}
\newtheorem{theorem}{Theorem}[section]
\newtheorem{corollary}[theorem]{Corollary}
\newtheorem{lemma}[theorem]{Lemma}

\newtheorem{proposition}[theorem]{Proposition}

\theoremstyle{definition}

\newtheorem{example}[theorem]{Example}
\newtheorem{remark}[theorem]{Remark}

\numberwithin{equation}{section}

\newcommand{\cU}{\mathcal U}
\newcommand{\cX}{\mathcal X}
\newcommand{\cY}{\mathcal Y}
\newcommand{\cZ}{\mathcal Z}

\renewcommand{\cH}{\mathcal H}

\newcommand{\Gm}{\mathbb G_m}

\newcommand{\PP}{\mathbb P}
\newcommand{\fP}{\mathfrak P}

\DeclareMathOperator{\Aut}{Aut}

\newcommand{\bk}{\Bbbk}

\newcommand{\oh}{\mathcal{O}}

\newcommand{\bG}{\mathbb{G}}
\newcommand{\bZ}{\mathbb{Z}}

\newcommand{\cP}{\mathcal{P}}
\newcommand{\bP}{\mathbb{P}}
\newcommand{\cM}{\mathcal{M}}
\newcommand{\cQ}{\mathcal{Q}}
\newcommand{\cI}{\mathcal{I}}
\newcommand{\cW}{\mathcal{W}}
\newcommand{\base}{\Bbbk}
\newcommand{\epf}{\qed \vspace{+10pt}}

\newcommand{\tensor}{\otimes}

\DeclareMathOperator{\Spec}{Spec}
\DeclareMathOperator{\Stab}{Stab}
\DeclareMathOperator{\id}{id}
\DeclareMathOperator{\im}{im}

\DeclareMathOperator{\cms}{cms}
\DeclareMathOperator{\triv}{triv}
\DeclareMathOperator{\GL}{GL}

\newcommand{\co}{\colon}
\renewcommand{\bar}{\overline}

\newcommand{\uIsom}{\underline{{\rm Isom}}}

\newcommand{\into}{\hookrightarrow} 
\renewcommand{\xto}{\xrightarrow} 
\newcommand{\iso}{\stackrel{\sim}{\to}}
\usepackage{stmaryrd}
\newcommand{\rigidify}{ \hspace{-.1cm} \fatslash \hspace{.10cm}}

\usepackage{pbox}

\begin{document}

\title{Structure results for torus fixed loci}
\date{\today}
\author[J. Alper]{Jarod Alper}
\address[J. Alper]{Department of Mathematics\\
  University of Washington\\
  Box 354350 \\
  Seattle, WA 98195--4350\\
  U.S.A.}
\email{jarod@uw.edu}
\author[F. Janda]{Felix Janda}
\address[F. Janda]{Department of Mathematics\\
  University of Illinois Urbana--Champaign\\
  Urbana, IL 61801\\
  U.S.A.}
\email{fjanda@illinois.edu}
\begin{abstract}
  Motivated by localization theorems on moduli spaces, we prove a structural classification of Deligne--Mumford stacks with an action of a torus where the induced action on the coarse moduli space is trivial.  We also establish a general local structure theorem for morphisms of algebraic stacks.
\end{abstract}

\maketitle

\section{Introduction}

The main goal of this paper is to provide a description of 
Deligne--Mumford stacks $X$ with a torus action where 
the induced action on the coarse moduli space $X_{\cms}$ is trivial. 
Such actions arise naturally in moduli theory as the fixed locus of a torus action on a Deligne--Mumford stack (see \Cref{rmk:moduli}).  The motivation for this paper is that an explicit understanding of such actions could be useful for applications of localization theorems.  

A prototypical example is the non-trivial $\bG_m$-action on $B \mu_r$ arising from the Kummer sequence $0 \to \mu_r \to T=\bG_m \xto{r} \bG_m \to 0$:  the induced map $B\mu_r \to BT$ is a $\bG_m$-torsor as it's the base change of $\Spec \bk \to B \bG_m$.  The main theorem below states essentially that every $\bG_m$-action on a Deligne--Mumford stack $X$, which is trivial on $X_{\cms}$, is built from this nontrivial action on $B \mu_r$ for some $r$. 

\subsection{The main theorem}
We work over an algebraically closed field $\bk$.

\begin{theorem} \label{thm:main}
      Let $\cX = [X/\bG_m^n] \to B \bG_m^n$ be a separated, finite type, and relatively tame Deligne--Mumford morphism of algebraic stacks (i.e., $X$ is a separated, finite type, and tame Deligne--Mumford stack).  Assume that $\cX$ is reduced and connected, and that the $\bG_m^n$-action on the coarse moduli space $X_{\cms}$ of $X$ is trivial. 
      \begin{enumerate}[(1)]
        \item  \label{thm:main1} 
          There is a diagonalizable group $\mu_r := \mu_{r_1} \times \cdots \times \mu_{r_n}$ for a unique tuple $r = (r_1, \ldots, r_n)$ of integers with $r_1 | r_2 | \cdots | r_n$ and a central closed subgroup $\mu_{r,\cX} \subset I_{\cX/B\bG_m^n}$ of the relative inertia stack restricting to a subgroup $\mu_{r,X} \subset I_X$ such that the $\bG_m^n$-action on $X$ descends to a trivial $\bG_m^n$-action on the Deligne--Mumford stack $Y := X \rigidify \mu_{r,X}$ and there is a factorization 
        $$\xymatrix{
          \cX \ar[rd]  \ar[r]  
            & \cX \rigidify \mu_{r,\cX}  \cong Y \times B\bG_m^n \ar[d]\\
            & X_{\cms} \times B\bG_m^n
        }$$
        which satisfies the following universal property:
        if $\cX \to Y' \times B \bG_m^n \to X_{\cms} \times B \bG_m^n$ is another factorization where $Y'$ is a Deligne--Mumford stack, then there is a morphism $g \co Y \to Y'$ unique up to unique isomorphism fitting in the diagram
            $$\xymatrix{
              \cX \ar[r] \ar[rd]
                & Y \times B\bG_m^n \ar@{-->}[d]^{g \times \id} \ar[r]
                & X_{\cms} \times B\bG_m^n \\
                & Y' \times B\bG_m^n . \ar[ur] 
            }$$
        \item  \label{thm:main2}
        Consider the exact sequence $1 \to \mu_r \to T \to \bG_m^n \to 1$ where $T = \bG_m^n \to \bG_m^n$ is defined by $(t_1, \ldots, t_n) \mapsto (t_1^{r_1}, \ldots, t_n^{r_n})$, 
        and let 
        $BT \to B\bG_m^n$ be the induced $\mu_r$-gerbe.  
          The $\mu_r \times \mu_r$-gerbe $X \times BT \to Y \times B \bG_m^n$ factors as 
          $$X \times BT\to \cX \to Y \times B \bG_m^n.$$
          In particular, $\cX \cong p_1^* X \wedge^{\mu_r} p_2^* BT$ over $Y \times B\bG_m^n$.
        \item \label{thm:main3} 
          There is a central closed fppf subgroup $T \times \cX \subset I_{\cX}$ and a factorization
          $$\xymatrix{
            \cX \ar[r] \ar[rd]
              & \cX \rigidify \mu_{r,\cX} \cong Y \times B \bG_m^n \ar[d]^{p_1} \\
              & \cX \rigidify T \cong Y
            }$$
            and $\cX \rigidify T$ is identified with the rigidification $(\cX \rigidify \mu_{r,\cX}) \rigidify \bG_m^n$ by $\bG_m^n = T / \mu_r$.
            Furthermore, $\cX \to Y$ is the $T$-gerbe associated
            to the $\mu_r$-gerbe $X \to Y$ with respect to the
            canonical inclusion of group schemes $\mu_r \subset T$.
      \end{enumerate}
  \end{theorem}

\subsection{Special cases and our proof strategy}

Given the technical nature of the statement and the subtleties arising from torus actions on Deligne--Mumford stack, we will try to unpack the conclusion in a series of remarks and examples.
\begin{remark}[Torus equivariant interpretation]
    Part \eqref{thm:main1} asserts that $Y$ is the largest Deligne--Mumford stack with trivial $\bG_m^n$-action factoring $\bG_m^n$-equivariantly as $X \to Y =X\rigidify \mu_{r,X} \to X_{\cms}$.  Note that there is an identification $X_{\cms} \cong Y_{\cms}$ of coarse moduli spaces. 
    Note also that the assumption that the $\bG_m^n$-action on $X_{\cms}$ is trivial translates to the condition that the relative coarse moduli space of $\cX \to B \bG_m^n$ is $X_{\cms} \times B\bG_m^n$.

    For part \eqref{thm:main2}, let $X_{\triv}$ be the $\mu_r$-gerbe over $Y$ which is isomorphic to $X$ as Deligne--Mumford stacks over $Y$ but is given the trivial $\bG_m^n$-action.  Then the $\bG_m^n$-equivariant $\mu_r$-gerbe $X \to Y$ is obtained from the possibly \emph{non-trivial} $\mu_r$-gerbe $X_{\triv} \to Y$ with \emph{trivial} $\bG_m^n$-action
    by twisting by the \emph{trivial} $\mu_r$-gerbe $B\mu_{r,Y} \to Y$ with \emph{non-trivial} $\bG_m^n$-action, i.e. there is a $\bG_m^n$-equivariant isomorphism of $\mu_r$-gerbes
    $$X \cong X_{\triv} \wedge^{\mu_r} B\mu_{r,Y}.$$
    While the $\bG_m^n$-action on $X$ may be non-trivial, part \eqref{thm:main2} also asserts that there is a smallest integer $r$ such that the action becomes trivial after the reparameterization $T=\bG_m^n \to \bG_m^n$ defined by $(t_1, \ldots, t_n) \mapsto (t_1^{r_1}, \ldots, t_n^{r_n})$.  
\end{remark}

\begin{example}[Case of a classifying stack]
  \label{ex:classifying-stack}
  The crucial example---both to understand the conclusion of the main theorem and its proof---is a $\bG_m^n$-action on a classifying stack $X=BG$ of a finite group $G$.  Such an action corresponds to an extension 
  \begin{equation} \label{eq:extension}
    1 \to G \to \Gamma \to \bG_m^n \to 1
  \end{equation}
  of algebraic groups (see \Cref{ex:action-on-BG}) where $\cX = [X/\bG_m^n] = B \Gamma$. 
  We let $T = \Gamma^0 \subset \Gamma$ be the identity component.   As $T$ is isomorphic to $\bG_m^n$ (see \S\ref{ss:extension}), the intersection $T \cap G$ is isomorphic to $\mu_r$ for a unique tuple of integers $r=(r_1, \ldots, r_n)$ with $r_1 | \cdots | r_n$. 
  By \Cref{lem:split,lem:central}, $T$ and $\mu_r$ are central in $\Gamma$, and there is an identification $\bar{\Gamma} \cong \bar{G} \times \bG_m^n$ where  $\bar{G} := G / \mu_r$ and $\bar{\Gamma} := \Gamma / \mu_r$.    Part \eqref{thm:main1} of \Cref{thm:main} is the identifications 
  $$\cX \rigidify \mu_r \cong B \bar{\Gamma} \cong \underbrace{B\bar{G} \times B\bG_m^n}_{Y \times B \bG_m^n}$$ 
  over $B \bG_m^n$ while \eqref{thm:main2} is the factorization
  $$\underbrace{BG  \times BT}_{X \times BT} 
    \to \underbrace{B \Gamma}_{\cX} 
    \to \underbrace{B \bar{\Gamma}}_{B \bar{G} \times B \bG_m^n}.
  $$
  and \eqref{thm:main3} is the factorization 
  $$
    \underbrace{B \Gamma}_{\cX} \to 
      \underbrace{B \bar{\Gamma}}_{\cX \rigidify \mu_r} \to  
      \underbrace{B \bar{G}}_{\cX \rigidify T}.
  $$
\end{example}

\begin{example}[Case of a quotient stack]
  \label{ex:quotient-stack}
  Let $G$ be a finite group acting on a connected, reduced, and separated algebraic space $U$ of finite type over $\base$.  Suppose that $X:=[U/G] \to BG$ is a $\bG_m^n$-equivariant morphism such that the induced $\bG_m^n$-action on the quotient algebraic space $U/G = [U/G]_{\cms}$ is trivial.  Then $B\Gamma = [BG/\bG_m^n]$ for an extension $\Gamma$ of $\bG_m^n$ by a finite group $G$ as in \eqref{eq:extension}.  We set $\cX = [U/\Gamma]$ and we define $T=\Gamma^0$, $\mu_r = G \cap T$, $\bar{G} = G/\mu_r$, and $\bar{\Gamma} = \Gamma/\mu_r$ as above. 

  Since $U$ is reduced and $U \to U/G$ is quasi-finite and equivariant under $T \to \bG_m^n$, the action of $T$ (and thus $\mu_r$) on $U$ must be trivial.  This gives a central subgroup $\mu_{r,U}$ of the stabilizer group scheme of the action of $G$ on $U$ and thus a central subgroup $\mu_{r,\cX} \subset I_{\cX/B\bG_m^n}$ of the relative inertia.    
  Then \eqref{thm:main1}-\eqref{thm:main3}  of \Cref{thm:main} are expressed with the factorization 
  $$\underbrace{ [U/G] \times BT \cong [U/ (G \times T)]}_{X \times B T} 
    \to \underbrace{[U / \Gamma]}_{\cX} 
    \to \underbrace{[U/ \bar{\Gamma}] \cong [U/\bar{G}] \times B \bG_m^n}_{\cX\rigidify \mu_r} 
    \to \underbrace{[U/\bar{G}]}_{\cX\rigidify T}.$$
\end{example}

\begin{remark}[Root gerbes]
  \label{rmk:root-gerbes}
  In the case of a $\bG_m$-action, the $\mu_r$-gerbe $X \to Y$ will
  sometimes be a root gerbe and correspondingly so will
  $\cX \to Y \times B \bG_m$.
  This occurs for instance in \Cref{ss:moduli-example1} below.
  Recall that the $r$th root gerbe $\sqrt[r]{(X,L)}$ of a line bundle
  $L$ on $X$ is defined as the fiber product of the map
  $X \to B \bG_m$ classifying $L$ and the $r$th power map
  $B \bG_m \to B \bG_m$. 
  In fact, $X \to Y$ is a root gerbe if and only if its cohomology
  class in $H^2(Y_{\rm et}, \mu_r)$ is the image of a line bundle
  $[L] \in H^1(Y_{\rm et}, \bG_m)$ under the boundary map induced from the
  Kummer sequence $1 \to \mu_r \to \bG_m \to \bG_m \to 1$.
  If $X \to Y$ is a root gerbe corresponding to a line bundle $L$ on
  $Y$, then the root gerbe corresponding to $\cX \to Y \times B \bG_m$
  is $\cX \cong \sqrt[r]{(Y \times B\bG_m, L \boxtimes [1])}$ where
  $[1]$ is the weight-$1$ representation of $\bG_m$.
\end{remark}

\begin{remark}
  \label{rmk:moduli}
  Torus actions and fixed loci are ubiquitous in moduli theory.
  In fact, localization with respect to a torus action on the moduli
  space of stable maps as in \cite{GrPa99} is one of the main
  techniques in Gromov--Witten theory, and \Cref{thm:main} describes
  the equivariant geometry of the fixed loci in detail.
  In \Cref{ss:moduli-examples}, we give two examples illustrating
  \Cref{thm:main} for moduli spaces of stable maps.
\end{remark}

\begin{remark}[Strategy of proof in the general case]
    Our strategy to handle the general case of an algebraic stack $\cX=[X/\bG_m^n]$ over $B \bG_m^n$ is to reduce to the case where $X=[U/G]$ is a quotient by a finite group.  To this end, 
    we prove a general \'etale local structure theorem of morphisms $\cX \to \cY$ (see \Cref{thm:local-structure-morphisms}).  When $\cY = B \bG_m^n$, the theorem implies that any point of $X$ has a $\bG_m^n$-equivariant \'etale neighborhood 
    $W=[\Spec A/G] \to X$ or in other words an \'etale neighborhood $\cW=[W/\bG_m^n] \to \cX$ over $B\bG_m^n$. 
    We then show that the construction of the subgroup $\mu_{r, \cW} \subset I_{\cW/B\bG_m^n}$ descends a central subgroup $\mu_{r, \cX} \subset I_{\cX/B\bG_m^n}$ and that the factorizations in parts \eqref{thm:main1}-\eqref{thm:main3} can be checked \'etale locally.
\end{remark}

\begin{remark}[Necessity of hypotheses]
  The connectedness assumption in \Cref{thm:main} is not essential.
  If $X$ is disconnected, then there will still be a central subgroup
  of $I_{\cX/\bG_m}$ with the desired properties. 
  On each connected component, this subgroup will be isomorphic to
  $\mu_r$ for some tuple of integers $r$, but these integers will
  depend on the connected component. 
  
  On the other hand, the conclusion of \Cref{thm:main} may fail if
  $\cX$ is nonreduced. 
  Consider the affine scheme $U = \Spec(\bk[\epsilon]/\epsilon^2)$,
  and consider the action of $\Gamma = \Gm$ on $U$ defined by
  $\epsilon \mapsto t\epsilon$.
  The quotient stack $X =[U/\mu_2]$ by the subgroup
  $\mu_2 \subset \Gamma$ is a nonreduced, connected Deligne--Mumford
  stack with coarse moduli $U/\mu_2 = \Spec \bk$. 
  There is an induced action of $\bG_m = \Gamma/\mu_2$ on $X$, and the
  quotient $\cX = [X/\bG_m] = [U/\Gamma]$ is a nonreduced and
  connected algebraic stack over $B\bG_m$ via the structure morphism
  $\cX \to B\bG_m$ induced by the squaring map $\Gamma \to \bG_m$.  
  Note that $X \to \Spec \bk$ and $\cX \to B \bG_m$ are not
  $\mu_2$-gerbes.

  On the other hand, the theorem does extend to the nonreduced case if
  we modify the hypotheses:
\end{remark}

\begin{theorem}
  \label{thm:main-alt}
  Let $\cX = [X/\bG_m^n] \to B \bG_m^n$ be a quasi-separated, finite
  type, and relatively Deligne--Mumford morphism of algebraic stacks.
  Assume that the $\bG_m^n$-action on $\cX$ is trivial after a
  reparameterization of $\bG_m^n$.
  Then the conclusion of \Cref{thm:main} holds.
\end{theorem}

\begin{remark}
  In fact, as we will see in \Cref{rmk:triviality} below, when $\cX$
  is separated and reduced, a torus action trivial on coarse moduli is
  automatically also trivial after a reparameterization, and so
  \Cref{thm:main-alt} generalizes \Cref{thm:main}.
\end{remark}

\begin{remark}[Restriction to stabilizers]
  In the setting of \Cref{thm:main}, for every point $x \in X(\bk)$, there is a short exact sequence $1 \to \Stab_X(x) \to \Stab_{\cX}(x) \to \bG_m^n \to 1$.  The central subgroup $\mu_{r,\cX} \subset I_{\cX/\bG_m^n}$ restricts to the subgroup $\mu_r = \Stab_{\cX}(x)^0 \cap \Stab_X(x) \subset \Stab_X(x)$.
\end{remark}

\begin{remark}
  This paper was inspired and strengthens the second author's work
  \cite{CJR22P} (joint with Q.~Chen and Y.~Ruan) on punctured R-maps
  in the context of gauged linear sigma models (GLSM).
  In fact, as explained in Section~\ref{ss:glsm} below $\bG_m$-actions
  play a central role in the geometry of GLSM.
  In particular, the notation for the groups in the sequence
  \eqref{eq:extension} follows the one used in the work of
  Polishchuk--Vaintrob \cite[\S~3.2]{PoVa16}, and Fan--Jarvis--Ruan
  \cite[\S~2]{FJR18}.
\end{remark}

As $\mu$-gerbes play a central role in the structure of torus fixed loci of Deligne--Mumford stacks, it is interesting to study the behavior of invariants such as Picard and Chow groups under taking $\mu$-gerbes.
We hope to explore this in future work.

It would also be interesting to study the following generalizations.
Are there analogous results for (tame) Artin stacks with a torus action or Deligne--Mumford stacks with a non-abelian action?
Is there an analogous description for the attractor locus of a $\bG_m$-action?

\subsection{A local structure theorem for morphisms}

Our proof of \Cref{thm:main} uses a special case of the following theorem, which we believe is independently interesting.   It is established using techniques from \cite{ahr} and \cite{ahr2}.

\begin{theorem} \label{thm:local-structure-morphisms}
  \label{thm:luna-etale-slice-theorem-for-morphisms}
  Let $f \co \cX \to \cY$ be a morphism of quasi-separated algebraic stacks locally of finite type over $\bk$ with affine stabilizers.  Let $x \in \cX(\bk)$ and $y=f(x) \in \cY(\bk)$ be its image.  Assume that the stabilizers $G_x$ and $G_y$ are
  linearly reductive.  Then there exist
  \begin{enumerate}[(1)]
    \item affine schemes $\Spec A$ and $\Spec B$ with actions of $G_x$ and $G_y$, respectively;
    \item $\bk$-points $x' \in \Spec A$ and $y' \in \Spec B$; 
    \item a morphism $g \co \Spec A \to \Spec B$ of affine schemes equivariant with respect to $G_x \to G_y$; and 
    \item \'etale morphisms $([\Spec A / G_x], x') \to (\cX,x)$ and  $([\Spec B / G_y], y') \to (\cY,y)$ inducing isomorphisms of stabilizer groups at $x'$ and $y'$
  \end{enumerate}
  such that the diagram
  $$\xymatrix{
    BG_x \ar[d]
      & [\Spec A/G_x] \ar[r] \ar[d]^g \ar[l]
      & \cX \ar[d]^f \\
    BG_y
      & [\Spec B/G_y] \ar[r]  \ar[l]
      & \cY
  }$$
  is commutative.
  Moreover, if $\cX$ and $\cY$ have separated diagonal (resp. affine diagonal), then the morphisms $[\Spec A/G_x] \to \cX$ and $[\Spec B/G_y] \to \cY$ can be arranged to be representable (resp. affine).
\end{theorem}

A variant of this theorem was recently established in \cite[Thm.~2.1]{dCHN}.

\subsection*{Acknowledgements}

We thank Martin Bishop, Qile Chen, Giovanni Inchiostro, Max Lieblich, Rahul
Pandharipande, and Minseon Shin for useful discussions related to this work.

The first author was partially supported by NSF grant DMS-2100088. 
The second author was partially supported by NSF grants DMS-2239320,
DMS-2054830 and DMS-1638352.

\section{Examples and Application to gauged linear sigma models}

In this section, we further illustrate \Cref{thm:main} in examples in
moduli theory, and we explain an application to the geometry of gauged
linear sigma models.

\subsection{Examples in moduli theory}
\label{ss:moduli-examples}

\subsubsection{Stable maps to a Hirzebruch surface}
\label{ss:moduli-example1}

Let $S = \PP_{\PP^1}(\oh(n) \oplus \oh)$ be a Hirzebruch surface.
We may then consider the Deligne--Mumford stack
$M = \overline{\mathcal M}_{0, 0}(S, rF)$ of genus zero, unpointed
stable maps to $S$ whose curve class is $r$ times the fiber class $F$.
Equip $S$ with a $\bG_m$-action that scales the fibers, such that
$[S / \bG_m] \cong \PP_{\PP^1 \times B\bG_m}((\oh(n) \boxtimes
[1])\oplus \oh)$, where $[1]$ is the weight one representation of
$\bG_m$.
By post-composition, this equips $M$ with a $\bG_m$-action.

One of the fixed loci of this $\bG_m$-action corresponds to the closed
substack $X$ of ``Galois covers'' $C \to S$ that are $r$-fold covers
of a fiber of $S \to \PP^1$ totally ramified at two points in $C$.
We may identify $X$ with the root gerbe $\sqrt[r]{(\PP^1, \oh(n))}$,
where the associated map $X \to \PP^1$ records the image of $C$ in
$\PP^1$, and the universal $r$th root corresponds to the normal bundle
of $C$ at one of the ramification points.

The quotient $\cX = [X/\bG_m]$ may be identified with
$\sqrt[r]{(\PP^1 \times B\bG_m, \oh(n) \boxtimes [1])}$.
This illustrates \Cref{thm:main} \eqref{thm:main1}, \eqref{thm:main2}
with $Y = \PP^1$.

\subsubsection{Stable maps to projective space}

Consider the $\bG_m$-action on $\PP^N$ given by distinct weights $\lambda_0, \ldots, \lambda_N$, and the induced $\bG_m$-action on $\bar{\cM}_{g,n}(\bP^N, d)$. 
We illustrate \Cref{thm:main} for each connected component of the fixed locus $\bar{\cM}_{g,n}(\bP^N, d)^{\bG_m}$.  Following \cite[\S 4]{GrPa99}, each connected component corresponds to a marked graph $\Gamma$ as follows.  If $f \colon C \to \bP^N$ is a stable map fixed by $\bG_m$, then $f(C)$ is a $\bG_m$-invariant curve, and the image of every marked point, node, contracted component, and ramification point is a $\bG_m$-fixed point of $\bP^N$.  Since the weights are distinct, the $\bG_m$-fixed points of $\bP^N$ are the points $p_0, \ldots, p_N$ corresponding to the standard basis elements and the $\bG_m$-invariant curves are the lines connecting $p_i$ to $p_j$.  Each non-contracted component maps to one of these $\bG_m$-invariant curves and is ramified over two points, which implies that each such component is rational and that the map $f$ restricted to this component is determined by its degree. The graph $\Gamma$ corresponding to $f$ has vertices corresponding to the connected components of $f^{-1}(\{p_0, \ldots, p_N\})$ and has edges corresponding to the non-contracted component.  An edge $e$ is incident to a vertex $v$ if the corresponding components of $C$ intersect.  In addition, each edge $e$ is marked with the degree $d_e$ of the map from the non-contracted component to the line in $\bP^N$, and each vertex $v$ is marked with the arithmetic genus $g(v)$ of the corresponding component with the convention that $g(v) = 0$ if the component is a single point.  
There is also a labeling map $i \colon \rm{Vertices} \to \{0, \ldots, N\}$ defined by $f(v) = p_{i(v)}$.  
Finally, there are $n$ numbered legs corresponding to the $n$ marked points, with a leg attached to corresponding vertex.
    
The connected component of $\bar{\cM}_{g,n}(\bP^N, d)^{\bG_m}$ containing $f$ is isomorphic to 
$$\left[\bigg(\prod_v \bar{\cM}_{g(v), \rm{val}(v)} \bigg) / A\right],$$
where $A$ is a split extension
$$1 \to \prod_e \bZ/d_e \to A \to \Aut(\Gamma) \to 1.$$

We now determine the unique integer $r$ arising from \Cref{thm:main} for the connected component containing $f$.  For each edge $e$ connecting vertices $v_1$ and $v_2$, let $\lambda_{e,1}=\lambda_{i(v_1)}$ and $\lambda_{e,2}=\lambda_{i(v_2)}$.  We claim that 
\begin{equation} \label{eqn:r}
  r = \operatorname{lcm} \left\{ \frac{d_e}{\gcd(d_e, \lambda_{e,1}-\lambda_{e,2})} \ \bigg|\ e \text{ edge of } \Gamma \right\}.
\end{equation}
    
To see this, we consider the case of a single edge as the general case is similar and only notationally more complicated. Let $f \colon C \to \bP^N$ be a stable map with a single component $C_e$ that is not contracted and that maps $d$ to $1$ to the line connecting $p_i$ to $p_j$.  We may assume that the other pointed components of $C$ have no automorphisms.  
The stabilizer of $f$ in $[\bar{\cM}_{g,n}(\bP^N, d) / \bG_m]$ is
$$\mathrm{St}_f = \{(s,t) \, \mid \, s^d = t^{\lambda_{e,1} - \lambda_{e,2}} \} \subset \Aut(C) \times \bG_m,$$ 
where $s$ and $t$ are the coordinate of the $\bG_m$'s corresponding to the
automorphism group of the two-pointed $C_e$ and the two-pointed line in 
$\bP^N$.  Letting $h = \gcd(d, \lambda_{e,1} - \lambda_{e,2})$, the connected
component $\mathrm{St}^0_f$ of $\mathrm{St}_f$ is given by the vanishing of 
$s^{d/h} - t^{ (\lambda_{e,1} - \lambda_{e,2}) / h}$.  
The map $\mathrm{St}_f^{0} \to \bG_m$ is therefore $d/h$ to $1$, and thus $r = d/h$, which agrees with \eqref{eqn:r}.

\subsection{Relationship to Gauged Linear Sigma Models}
\label{ss:glsm}

Gauged linear sigma models are an important class of models for string
theory, initially considered by Witten \cite{Wi93}.
The theory has many aspects such as that it involves a reductive
``gauge'' group $G$ acting linearly on a vector space $V$, and a
``phase'', which corresponds to the choice of a GIT chamber for the
quotient $[V / G]$.
Most importantly for us, the theory has an additional ``R-symmetry''
corresponding to a $\bG_m$-action on $[V / G]$.

In \cite{FJR18}, a mathematical theory of enumerative invariants of
GLSMs has been constructed, involving a proper moduli stack and
virtual cycle.
Many recent advances in the mathematics of the higher genus Gromov--Witten theory
of complete intersections are based on gauged linear sigma models, see
for instance \cite{GJR18P} and \cite{CGL21}.
  
The geometry of the target of GLSM is a Deligne--Mumford stack
$\fP^\circ_\bk$ together with an action of $\bG_m$.
In the above setup, $\fP^\circ_\bk = [V / G]$; however, as explained
in \cite{CJR21}, it is possible to develop much of the theory for
Deligne--Mumford stacks $\fP^\circ_\bk$, which are not necessarily
quotients of this type.
Given this, the moduli objects considered in GLSM are, roughly
speaking, pre-stable curves $C$ together with a map%
\footnote{This assumes the $\epsilon = \infty$-stability condition --
  in general they may be quasi-maps.}
to $\fP^\circ_\bk$ twisted by the logarithmic dualizing line bundle
$\omega_C^{\log}$ according to the $\bG_m$-action.
One way to make this precise is using the notion of R-maps of
\cite{CJR21}: An R-map is a morphism
$f\colon C \to \fP^\circ := [\fP^\circ_\bk/\bG_m]$ together with a
2-commutative triangle
\begin{equation*}
  \xymatrix{
    & \fP^\circ \ar[d] \\
    C \ar[ur]^f \ar[r]^{\omega_C^{\log}} & B\bG_m,
  }
\end{equation*}
in which the lower horizontal arrow is given by the $\bG_m$-torsor
corresponding to $\omega_C^{\log}$.
The target $\fP^\circ_\bk$ of a GLSM is usually not proper, hence so
is the moduli space, and this causes technical difficulties in the
theory.
This non-properness problem was resolved in \cite{CJR21} via a
logarithmic compactification of the GLSM moduli space.
This involves choosing a $\bG_m$-equivariant compactification
$\fP_\bk$ of $\fP_\bk^\circ$ such that
$\fP_\bk \setminus \fP^\circ_\bk$ is a smooth, connected Cartier
divisor $\infty_\bk$ that is fixed under the $\bG_m$-action.
Furthermore, $\fP_\bk$ is equipped with a divisorial log structure
along $\infty_\bk$.

As an important part of understanding the structure of logarithmic
compactifications of GLSM moduli spaces, \cite{CJR22P} studies
punctured R-maps, which are analogous to the punctured maps of
Abramovich--Chen--Gross--Siebert \cite{ACGS20P}, to the boundary
divisor $\infty_\bk$.
In general, the stack underlying $\infty_\bk$ is a smooth
Deligne--Mumford stack with a $\bG_m$-action whose coarse moduli space
is projective and with trivial $\bG_m$-action.
In particular, $\infty_\bk$ satisfies the assumptions of
\Cref{thm:main}.
The logarithmic structure of $\infty_\bk$ is determined by a single
$\bG_m$-equivariant line bundle $\oh(\infty_\bk)$ on $\infty_\bk$.

This work allows removing an additional assumption on the target
$\infty_\bk$ as required in \cite{CJR22P}.
More specifically, the results in \cite[\S7,8]{CJR22P} related to the
behavior of punctured R-maps under forgetting a marking assume that
$\infty_\bk$ has the structure of a root gerbe as in
\Cref{rmk:root-gerbes} (see \cite[\S1.2.2]{CJR22P}).

\Cref{thm:main} may be used to remove this extra assumption.
More specifically, Part \eqref{thm:main3} of \Cref{thm:main} gives
rise to a morphism $BT \times \infty \to \infty$ fitting into a
diagram
\begin{equation*}
  \xymatrix{
    BT \times \infty \ar[r] \ar[d] & \infty \ar[d] \\
    BT \times B\bG_m \ar[r] & B\bG_m
  },
\end{equation*}
where the lower arrow is induced by $T \times \bG_m \to \bG_m$,
$(t, \lambda) \mapsto t^r \lambda$.
In \cite[\S7.2]{CJR22P}, the analogous structure is called a
``twist'', and is only constructed in the root stack setting.
Existence of a twist is crucial for understanding the structure of the
forgetful morphisms.
Thus, \Cref{thm:main} allows removing the root stack assumption
for the results of \cite{CJR22P}.

\section{Gerbes and torus actions on stacks}

\subsection{Gerbes and rigidification}
  We follow the conventions of \cite{giraud} for gerbes.  
  If $\mu$ is a finite diagonalizable group scheme, then a $\mu$-gerbe $X \to Y$ by definition is a banded $\mu$-gerbe, i.e. $X \to Y$ is a gerbe together with the data of isomorphisms $\psi_x \co \mu_T \to \underline{\Aut}(x)$ for every object $x \in X(T)$ such that for an isomorphism $\alpha \co x \to y$ over $T$, 
  $$
		\xymatrix{
			&	G|_{T} \ar[rd]^{\psi_y} \ar[ld]_{\psi_x} \\
		\underline{\Aut}_T(x) \ar[rr]^{{\rm Inn}_\alpha}		&		& \underline{\Aut}_T(y).
		}
	$$
	commutes, where ${\rm Inn}_\alpha(\tau) = \alpha \tau \alpha^{-1}$.  In other words, there is an isomorphism $\psi \co \mu_{X} \to I_{X/Y}$ of relative group schemes over $X$.

  For an algebraic stack $X$ and a closed fppf subgroup $\mu \subset I_X$, the rigidification $X \rigidify \mu$ is defined as the stackification over $({\rm Sch}/\bk)_{\rm fppf}$ of the prestack with the same objects as $X$ but where the set of morphisms between $a \in X(S)$ and $b \in X(T)$ over $f \co S \to T$ is ${\rm Mor}_{X}(a,f^*b)/\mu(S)$.  The rigidification $X \rigidify \mu$ is an algebraic stack and $X \to X \rigidify \mu$ is a $\mu$-gerbe. 

  The universal property of a $\mu$-rigidification $\cX \rigidify \mu$ is as follows:  if $\cX \to \cZ$ is a morphism such that the composition $\mu_{r} \into I_{\cX} \to I_{\cZ} \times_{\cZ} \cX$ is the trivial map of group schemes, then there is a map $\cX \rigidify \mu \to \cZ$ unique up to unique isomorphism filling in the diagram
  $$\xymatrix{
    \cX \ar[d] \ar[rd] \\
    \cX \rigidify \mu \ar@{-->}[r]
      & \cZ.
  }$$
  If $\cZ$ is Deligne--Mumford, then by \Cref{lem:trivial-map-of-group-schemes} below, it is enough to check that for every geometric point $x \in \cX(L)$ with image $z \in \cZ(L)$, there is an inclusion
  $$\mu_x \subset \ker(\Stab_{\cX}(x) \to \Stab_{\cZ}(z)).$$

  \begin{lemma} \label{lem:trivial-map-of-group-schemes}
    Let $H \to G$ be a morphism of group schemes, each finite type over a noetherian scheme $S$.  Suppose that $G \to S$ is unramified.  Then $H \to G$ is the trivial map if and only if for every geometric point $s \in S(L)$, the map $H_s \to G_s$ of group schemes over $L$ is the trivial map
  \end{lemma}
  
  \begin{proof}
    Since $G \to S$ is unramified, the identity section $S \to G$ is an open immersion.  As each geometric fiber $H_s$ factors through the identity section, so does $H$.  
  \end{proof}

  Given $\mu$-gerbes $X_1 \to Y$ and $X_2 \to Y$, the contracted product $X_1 \wedge^{\mu} X_2$ is the $\mu$-gerbe over $Y$ defined as the rigidification $(X_1 \times X_2) \rigidify \mu$ where $\mu \to I_{X_1 \times X_2}$ is the diagonal subgroup induced from the $\mu$-bands of $X_1$ and $X_2$.

\subsection{Group actions on stacks}
We follow \cite{romagny} for our conventions for group actions on stacks.  
A group action of an algebraic group $G$ on an algebraic stack $X$ is a morphism $\mu\co  G \times X \to X$ of stacks such that the usual diagrams
\[
    \begin{matrix}
        \xymatrix{
            G \times_S G \times_S X \ar[r]^{\hspace{0.5cm}\id_G \times \mu} \ar[d]_{m \times \id_X} 
                & G \times_S X \ar[d]^{\mu} \\
            G \times_S X \ar[r]^{\mu}       & G        
        }
    \end{matrix}
    \qquad
    \begin{matrix}
        \xymatrix{
            X \ar[r]^{(e, \id_X) \hspace{0.5cm}} \ar[rd]_{\id_X}       
                & G \times_S X \ar[d]^{\mu} \\
                & X
        } 
    \end{matrix}
\]
\emph{strictly} commute.  

\begin{example}[Actions on $BG$]
  \label{ex:action-on-BG}
  If $G$ is a finite group, then a $\bG_m^n$-action on $BG$ is equivalent to a short exact sequence
  $$1 \to G \to \Gamma \to \bG_m^n \to 1.$$ 
  The action is trivial if and only if the sequence splits if and only if the sequence splits trivially.  

  To see this, observe that such a short exact sequence induces a cartesian diagram
  $$\xymatrix{
    \bG_m^n \ar[r] \ar[d]
      & BG \ar[r] \ar[d] 
      & \Spec \bk \ar[d] \\
    \Spec \bk \ar[r]
      & B \Gamma \ar[r]
      & B \bG_m^n.
  }$$
  In particular, $BG \to B\Gamma$ is a $\bG_m^n$-torsor and $\bG_m^n$ acts on $BG$.  Conversely, the quotient stack $[BG/\bG_m^n]$ 
  has one point and is thus isomorphic to $B \Gamma$ for some algebraic group $\Gamma$.  The composition $BG \to B \Gamma \to B \bG_m^n$ induces a short exact sequence $1 \to G \to \Gamma \to \bG_m^n \to 1$.
\end{example}

\subsection{Fixed locus of a torus action}
The are several distinct ways to formulate the fixed locus for the action of a torus $T$ on a Deligne--Mumford stack $X$.  First, there is the stack of morphisms
$$\underline{\rm Mor}^T(\Spec \bk, X),$$
where an object over a $\bk$-scheme $S$ is the data of a $T$-equivariant map $S \to X$, where $S$ has the trivial action.  This is an algebraic stack and if $X$ has separated diagonal, then $\underline{\rm Mor}^T(\Spec \bk, X) \to X$ is a closed immersion.   For the nontrivial action of $\bG_m$ on $B\mu_n$ arising from the Kummer sequence, the stack  $\underline{\rm Mor}^T(\Spec \bk, B \mu_n)$ is empty.  Note that the cover $\Spec \base \to B \mu_n$ is not $\bG_m$-equivariant, but it becomes equivariant with respect to the reparameterization $\bG_m \to \bG_m$ given by $t \mapsto t^n$.  

We define the \emph{fixed stack} as the union of the closed substacks
$$X^T := \bigcup_{T' \to T} \underline{\rm Mor}^{T'}(\Spec \bk, X)$$
over all homomorphisms $T' = \bG_m^n \to \bG_m^n = T$ with finite kernel.  In the example of a nontrivial $\bG_m$ action on $B\mu_n$, the fixed locus is $B\mu_n$. 
See \cite[\S~5.4]{ahr}, \cite[\S~6.3]{Kr99}, and \cite[Appendix~A]{AKLPR22P} for a further discussion of fixed loci.

\begin{remark}[Trivial after reparameterization]
  \label{rmk:triviality}
  Consider the following three conditions for an action of a torus $T$ on a separated Deligne--Mumford stack $X$ of finite type over $\bk$:
  \begin{enumerate}[(a)]
    \item $X^T = X$, 
    \item there exists a reparameterization $T' \to T$ of tori such that the $T'$-action on $X$ is trivial,
    \item the induced $T$-action on $X_{\cms}$ is trivial.
  \end{enumerate}
  We claim that there are implications
  $$\text{(a)} \Longleftrightarrow \text{(b)} \Longrightarrow \text{(c)}, $$
  and that (c) $\Longrightarrow$ (b) if in addition $X$ is reduced.  The equivalence $\text{(a)} \Longleftrightarrow \text{(b)}$ follows from the definition of the fixed stack $X^T$.  Observe that the action of a torus $T$ on a separated algebraic space $Y$ is trivial if and only if it is trivial after a reparameterization $T' \to T$: indeed, this follows from flat descent as the stabilizer subgroup scheme $\Stab_X^T \subset T \times X$ pulls back to $\Stab_X^{T'} \subset T' \times X$.  Therefore, $\text{(b)} \Longrightarrow \text{(c)}$.

  Assuming (c), \cite[Proposition~5.32]{ahr} 
  implies that that there is a reparameterization $T' \to T$ such that $\underline{\rm Mor}^{T'}(\Spec \bk, X) \into X$ is a nilimmersion.  Therefore, if $X$ is reduced, $X^{T} = X$.   
\end{remark}

\subsection{Universal properties}
The goal of this subsection is to prove and elaborate on the universal property of the factorization 
$$\cX \to \cX \rigidify \mu_r \cong Y \times B \bG_m^n \to X_{\cms} \times B \bG_m^n$$
in Part \eqref{thm:main1} of \Cref{thm:main}.   

\begin{proposition} \label{prop:universal-properties}
  Under the hypotheses of \Cref{thm:main}, assume that there is a central diagonalizable closed subgroup scheme $\mu_{r, \cX} \subset I_{\cX/B\bG_m^n}$ with $r= (r_1, \ldots, r_n)$ and a factorization 
  \begin{equation}
    \label{eqn:factorization-universal-property}
    \cX \to \cX \rigidify \mu_{r,\cX} \cong Y \times B \bG_m^n \to X_{\cms} \times B \bG_m^n,
  \end{equation}
  where $Y$ is a Deligne--Mumford stack, 
  such that for every geometric point $x \in X(L)$, the restriction of $\mu_{r, \cX} \subset I_{\cX/B\bG_m^n}$ to $x$ is identified
  with the intersection $\Stab_X(x) \cap \Stab_{\cX}(x)^{0}$ inside $\Stab_{\cX}(x)$. 
  Then the factorization \eqref{eqn:factorization-universal-property} satisfies the following two universal properties:
  \begin{enumerate}[(a)] 
    \item \label{prop:universal-properties-a}
        if $\cX \to Y' \times B \bG_m^n \to X_{\cms} \times B \bG_m^n$ is another factorization where $Y'$ is a Deligne--Mumford stack, then there is a morphism $g \co Y \to Y'$ unique up to unique isomorphism fitting in the diagram
      $$\xymatrix{
        \cX \ar[r] \ar[rd]
          & Y \times B\bG_m^n \ar@{-->}[d]^{g \times \id} \ar[r]
          & X_{\cms} \times B\bG_m^n \\
          & Y' \times B\bG_m^n . \ar[ur] 
      }$$
    \item \label{prop:universal-properties-b} 
      if $\cH \subset I_{\cX/B\bG_m^n}$ is a closed fppf subgroup such that $\cX \rigidify \cH \cong Y' \times B\bG_m^n$ over $X_{\cms} \times B\bG_m^n$ for some Deligne--Mumford stack $Y'$, then $\mu_{r,\cX} \subset \cH$. 
  \end{enumerate}
\end{proposition}

\begin{proof}
  For \eqref{prop:universal-properties-a}, by the universal property of the rigidification and \Cref{lem:trivial-map-of-group-schemes},
  it suffices to show that for every geometric point $x \in \cX(L)$ with image $y' \in Y'(L)$, the induced map $\mu_r \into \Stab_{\cX}(x) \to \Stab_{Y'}(y')$ is trivial.
  Consider the group $T = \Stab_{\cX}(x)^0$, which is necessarily
  isomorphic to $\bG_m^n$ (see \S\ref{ss:extension}).
  Since $\Stab_{Y'}(y')$ is finite, the morphism
  $T \to \Stab_{Y'}(y')$ is necessarily trivial.
  Since by assumption, $T$ contains $\mu_r$, we conclude the proof of
  Part \eqref{prop:universal-properties-a}.

  The same argument can be used to show \eqref{prop:universal-properties-b}.  Alternatively, we can see \eqref{prop:universal-properties-b} as a consequence of \eqref{prop:universal-properties-a} since it implies the existence of the dotted arrow
  $$\xymatrix{
    \cX \ar[r] \ar[rd]
        & \cX \rigidify \mu_r \ar@{-->}[d] \ar[r]
        & X_{\cms} \times B \bG_m^n\\
        & \cX \rigidify \cH \ar[ur], 
  }$$
  which in turn implies the inclusion $\mu_{r,\cX} \subset \cH$.
\end{proof}

\begin{remark}
  The universal property of \eqref{prop:universal-properties-a} does not hold for an arbitrary algebraic stack $Y'$. For instance, for the action of $\bG_m$ on $X = B\mu_r$ arising from the Kummer sequence where
  $Y = \Spec \base$ and the $\mu_r$-rigidification $p \co \cX = BT \to B \bG_m$, the map $(\id, p) \co \cX \to BT \times B \bG_m$ does not factor through $p \co \cX \to B \bG_m$.

  We also point out that while the subgroup $\mu_{r,\cX} \subset I_{\cX/B\bG_m^n}$ is central, the universal property \eqref{prop:universal-properties-b} does not require $\cH$ to be.  
\end{remark}

\section{Extensions of tori and the quotient case}

In this section, we prove a description of extensions of tori by finite groups, which we then apply to prove \Cref{thm:main} and \Cref{thm:main-alt} in the cases of classifying and quotient stacks.

\subsection{Extensions of $\bG_m^n$ by finite groups}
\label{ss:extension}
Let $\Gamma$ be an algebraic group over $\bk$ that is a group extension of $\bG_m^n$
by a finite group $G$:
\begin{equation}
  \label{eq:ext-gamma}
  1 \to G \to \Gamma \to \bG_m^n \to 1,
\end{equation}
which by \Cref{ex:action-on-BG} corresponds to a $\bG_m^n$-action on $BG$. 
The connected component $\Gamma^0 \subset \Gamma$ of the identity is a reductive algebraic group of dimension $n$.  Since we have a surjection $\Gamma^0 \to \bG_m^n$, by
\cite[Cor.~1 of Prop.~11.14]{borel} there is a torus $T \subset \Gamma^0$ such the composition $T \into \Gamma^0 \to \bG_m^n$ is an isogeny.  Since $\dim \Gamma^0 = n$, we conclude that $\Gamma^0 = T$ and that $T$ is an $n$-dimensional torus. 

The composition $T \into \Gamma \to \bG_m^n$ fits into an exact sequence
\begin{equation}
  \label{eq:ext-T}
  1 \to \mu_r \to T \to \bG_m^n \to 1,
\end{equation}
for a diagonalizable group scheme $\mu_r$.  The map $T \to \bG_m^n$ corresponds to an injective map of free $\bZ$-modules of rank $n$.  By the structure theorem of submodules of $\bZ^r$, we can choose a basis such that $T \to \bG_m^n$ is given by $(t_1, \ldots, t_n) \mapsto (t_1^{r_1}, \ldots, t_n^{r_n})$ for a unique tuple $r = (r_1, \ldots, r_n)$ of integers with $r_1 | r_2 | \cdots | r_n$.  This implies that
$$ \mu_r = \mu_{r_1} \times \cdots \times \mu_{r_n}.$$
By uniqueness of $T$, the subgroup $\mu_r \subset \Gamma$ is
characteristic, and hence normal.

The inclusion $\mu_r \into T \into \Gamma$ factors through $G$ yielding an exact sequence
\begin{equation} \label{eqn:gamma-mod-mur}
  1 \to \bar{G} \to \bar{\Gamma} \to \bG_m^n \to 1,
\end{equation}
where $\bar{G} = G/\mu_r$ and $\bar{\Gamma} = \Gamma/\mu_r$. 

\begin{lemma} \label{lem:split}
  The sequence \eqref{eqn:gamma-mod-mur} is trivially split, 
  i.e. $\bar{\Gamma} \cong \bar{G} \times \bG_m^n$.
\end{lemma}

\begin{proof}
    The inclusion $T \into \Gamma$ descends to a homomorphism 
    $$T/\mu_{r} = \bG_m^n \to \bar{\Gamma} = \Gamma/\mu_r$$
    splitting \eqref{eqn:gamma-mod-mur}.  The lemma follows as there are no 
    non-trivial homomorphisms $\bG_m^n \to \Aut(\bar{G})$
\end{proof}

In other words, the lemma gives an exact sequence
\begin{equation*}
  1 \to \mu_r \to \Gamma \to \bar{G} \times \bG_m^n \to 1.
\end{equation*}
Moreover,  since the 
homomorphism $G \times T \to \Gamma$ is surjective, there is an exact sequence
\begin{equation}
  \label{eq:Gamma-quotient}
  1 \to \mu_r \to G \times T \to \Gamma \to 1
\end{equation}
induced by the inclusions $G \subset \Gamma$ and $T \subset \Gamma$.

\begin{lemma} \label{lem:central}
  The subgroup $T \subset \Gamma$ is central with quotient $\Gamma/T \cong \bar{G}$.
\end{lemma}
\begin{proof}
  We first show that the sequence
  \begin{equation}
    \label{eq:lesTH}
    1 \to T \xrightarrow{i} \Gamma \xrightarrow{j} \bar{G} \to 1
  \end{equation}
  is exact, where $i$ is the inclusion, and $j$ is the composition
  $\Gamma \to \bar{\Gamma} \to \bar{G}$.
  Clearly, $i$ is injective and $j$ is surjective.
  Now consider the commutative diagram
  \begin{equation*}
    \xymatrix{
      1 \ar[r] & T \ar[r]^i \ar[d]^p & \Gamma \ar[d]^{q} \ar[dr]^{j} \ar[r] & K \ar[r] \ar[d]^k & 1\\
      1 \ar[r] & \bG_m^n \ar[r] & \bar{\Gamma} \ar[r] & \bar{G} \ar[r] & 1
    }
  \end{equation*}
  with exact rows.
  Since $p$ is surjective, and $\ker(p) \to \ker(q)$ is an
  isomorphism, the snake lemma implies that $k$ is an isomorphism,
  yielding the exact sequence \eqref{eq:lesTH}.

Since $T$ is normal in $\Gamma$, conjugation induces a homomorphism $\Gamma \to \Aut(T) = \GL_n(\bZ)$. Similarly, since $\bG_m^n$ is normal in $\bar{\Gamma}$, there is a homomorphism $\bar{\Gamma} \to \Aut(\bG_m) = \GL_n(\bZ)$.  The subgroup $\mu \subset T$ is characteristic, and there is an induced map $\Aut(T) \to \Aut(\bG_m^n)$, which must be an isomorphism.  This gives a commutative diagram
$$\xymatrix{
    \Gamma \ar[r] \ar[d] 
      & \Aut(T) \ar[d]^{\rotatebox{90}{$\sim$}} \\
    \bar{\Gamma} \ar[r]
      & \Aut(\bG_m^n).
  }$$
  Since $\bar{\Gamma} \cong \bar{G} \times \bG_m^n$, we conclude that $\bar{\Gamma} \to \Aut(\bG_m^n)$ is trivial and thus so is $\Gamma \to \Aut(T)$.
\end{proof}

\begin{proposition} \label{prop:classifying-stack-case}
  \Cref{thm:main} holds for $\cX = B \Gamma$.
\end{proposition}

\begin{proof}
  \Cref{lem:central} gives central subgroups $\mu_r \subset \Gamma$ and $T \subset \Gamma$ with isomorphisms
  $$
    B\Gamma \rigidify \mu_r \cong B \bar{\Gamma} 
    \quad \text{and} \quad
    B\Gamma \rigidify T \cong B \bar{G} 
  $$
  where as above $\bar{\Gamma} = \Gamma/\mu_r$ and $\bar{G} = G/\mu_r$.  \Cref{lem:split} gives an identification of $\bar{\Gamma} \cong \bar{G} \times \bG_m^n$.
  Defining $Y:=B \bar{G}$, the factorization $B \Gamma \to B \Gamma \rigidify \mu_r \cong Y \times B \bG_m^n \to B \bG_m^n$ satisfies the universal property by \Cref{prop:universal-properties}, yielding Part~\eqref{thm:main1}.
  Part~\eqref{thm:main2} follows from the exact sequence \eqref{eq:Gamma-quotient}, while Part~\eqref{thm:main3} follows from the identification $B\Gamma \rigidify T \cong B \bar{G} = Y$ and \eqref{eq:Gamma-quotient}.
\end{proof}

\subsection{Quotient case}

The case of a quotient stack follows easily from the case of a
classifying stack.
We will use the group-theoretic notation from
Section~\ref{ss:extension}.

\begin{proposition}
  \label{prop:quotient-stack-case-alt}
  \Cref{thm:main-alt} holds for $\cX = [U/\Gamma]$ where $U$ is a
  connected and quasi-separated algebraic space $U$ of finite type
  over $\base$.
\end{proposition}

\begin{proof}

  By the hypotheses of \Cref{thm:main-alt}, $X:=[U/G] \to BG$ is a
  $\bG_m^n$-equivariant morphism such that the $\bG_m^n$-action on $X$
  becomes trivial after a reparameterization $T'$ of $\bG_m^n$.
  We may assume without loss generality that $T' \to \bG_m^n$ factors
  through $T$.
  Since the action of $T'$ on $X$ is trivial, we have
  $[U / (\Gamma \times_{\bG_m^n} T')] \cong [U / G] \times BT'$ over
  $BT'$.
  We claim that since $p \colon U \to X$ is finite \'etale, the $T'$ 
  action on $U$ is trivial.    
  Indeed, as the $T'$ action on $X$ is trivial, 
  $\mathcal{A} := p_* \oh_{U}$ can be written as the direct sum 
  $\mathcal{A}_0 \oplus \mathcal{A}'$, where $\mathcal{A}_0$ is the weight 
  0 subspace and $\mathcal{A}'$ is the direct sum of the non-trivial weight 
  subspaces. If $I \subset \mathcal{O}_X$ denotes the nilradical, then 
  $I \mathcal{O}_U$ is the nilradical of $U$ (since $U \to X$ is étale) and $I \mathcal{A} = p_*(I \mathcal{O}_U)$
  is the nilradical of $\mathcal{A}$. 
  As $T'$ acts trivially on $U_{\rm red}$, 
  $\mathcal{A}' \subset I \mathcal{A} = I \mathcal{A}_0 \oplus I \mathcal{A}'$, 
  and it follows from weight considerations 
  that $\mathcal{A'} \subset I \mathcal{A'}$.  
  As $\mathcal{A}'$ is a coherent $\mathcal{O}_X$-module (since $U \to X$ is finite), 
  Nakayama's Lemma implies that
  $\mathcal{A}'=0$.

  This implies that the action of $T$ on $U$ is also trivial.
  Therefore, the $\mu_r$-action on $U$ is trivial and we have an isomorphism
  $\cX \rigidify \mu_r \cong [U/\bar{\Gamma}]$.  Using the isomorphism $\bar{\Gamma} \cong \bar{G} \times \bG_m^n$ from \Cref{lem:split}, we have further identifications
  $$\cX \rigidify \mu_r \cong [U/\bar{\Gamma}] \cong [U / (\bar{G} \times \bG_m^n)] \cong [U/\bar{G}] \times B \bG_m^n,$$
  where in the final equivalence we have used that the $\bG_m^n$-action on $U$ must be trivial (as the $T$-action is trivial).

  For each geometric point $x \in X(L)$, restricting the $\bG_m^n$-torsor $X \to \cX$ to residual gerbes at $x$ induces a $\bG_m^n$-torsor $BG_x \to B \Gamma_x$, where $G_x = \Stab_X(x) \subset G$ and $\Gamma_x = \Stab_{\cX}(x)$. As $\Gamma_x^0 = \Gamma^0$, we have the identification $\mu_r = G_x \cap \Gamma_x^0$.  By \Cref{prop:universal-properties}, the factorization 
  $$[U/\Gamma] \to [U/\bar{\Gamma}] = [U/\bar{G}] \times B \bG_m^n \to U/G \times B \bG_m^n$$ satisfies the universal property, yielding part \eqref{thm:main1}.
  As in the case of the classifying stack, parts \eqref{thm:main2} and \eqref{thm:main3} follow from the exact sequence \eqref{eq:Gamma-quotient} and the identification $B\Gamma \rigidify T \cong B \bar{G}$. 
\end{proof}

\begin{corollary} \label{prop:quotient-stack-case}
  \Cref{thm:main} holds for $\cX = [U/\Gamma]$ where $U$ is a connected, reduced, and separated algebraic space $U$ of finite type over $\base$. 
\end{corollary}

\begin{proof}
  By \Cref{rmk:triviality}, there is a reparameterization $T' \to T$ of tori such that the $T'$-action on $X = [U/G]$ is trivial.  Therefore, the result follows from \Cref{prop:quotient-stack-case-alt}.  Alternatively, we can argue similarly to \Cref{prop:quotient-stack-case-alt}. By the hypotheses of \Cref{thm:main}, $X=[U/G] \to BG$ is a $\bG_m^n$-equivariant morphism such that the induced $\bG_m^n$-action on the quotient algebraic space $U/G = [U/G]_{\cms}$ is trivial.  The map $U \to U/G$ is quasi-finite and equivariant under $T \to \bG_m^n$.  Since $U$ is reduced, the $T$-action on $U$ must also be trivial.  From here, the proof proceeds as in \Cref{prop:quotient-stack-case-alt}.
\end{proof}

\section{A $\bG_m$-equivariant Luna \'etale slice theorem}

In this section, we prove an \'etale local structure theorem for morphisms (\Cref{thm:luna-etale-slice-theorem-for-morphisms}), from which we deduce a torus-equivariant Luna \'etale slice theorem (\Cref{cor:equivariant-luna}), which in turn is applied to prove the general case of \Cref{thm:main}.

\subsection{\'Etale local structure of morphisms}

\begin{proof}[Proof of \Cref{thm:luna-etale-slice-theorem-for-morphisms}]
    Applying \cite[Thm. 1.1]{ahr} to $(\cY,y)$ yields an \'etale morphism 
    \begin{equation} \label{eqn:etale-chart1}
      ([\Spec B / G_y], y') \to (\cY,y)
    \end{equation}
    inducing an isomorphism of stabilizer groups at $y'$. Applying loc. cit. 
    again to the base change 
    $[\Spec B / G_y] \times_{\cY} \cX$ yields an \'etale morphism 
    \begin{equation} \label{eqn:etale-chart2}
      ([\Spec A / G_x], x') \to ([\Spec B / G_y] \times_{\cY} \cX,(y',x))
    \end{equation}
    inducing an isomorphism of stabilizer groups at $x'$. 
    The composition $[\Spec A/G_x] \to [\Spec B / G_y] \times_{\cY} \cX \to \cX$ is also \'etale and induces an isomorphism of stabilizer groups at $x'$.  It remains to verify that the morphism $[\Spec A/G_x] \to [\Spec B/G_y]$ is induced by a morphism $\Spec A \to \Spec B$ equivariant with respect to $G_x \to G_y$, or equivalently that the diagram
    $$\xymatrix{
      [\Spec A/G_x] \ar[r] \ar[d]   & BG_x \ar[d] \\
      [\Spec B/G_y] \ar[r]          & BG_y
    }$$
    is commutative.

    Let $p, q \co [\Spec A/G_x] \to BG_y$ denote the two composition morphisms, 
    and let $\cP$ and $\cQ$ denote the corresponding $G_y$-torsors over $[\Spec A/G_x]$.  Since \eqref{eqn:etale-chart1} and \eqref{eqn:etale-chart2} induce isomorphism of stabilizers at $y'$ and $x'$, there is a 2-isomorphism $p|_{BG_{x'}} \iso q|_{BG_{x'}}$ of the restrictions 
    of $p$ and $q$ along the inclusion $BG_{x'} \to [\Spec A/G_x]$ of the residual 
    gerbe of $x'$, and therefore an isomorphism 
    $\alpha_0 \co \cP|_{BG_{x'}} \to \cQ|_{BG_{x'}}$ of $G_y$-torsors.  The stack 
    $\cI := \uIsom_{[\Spec A/G_x]}(\cP, \cQ)$ parameterizing isomorphisms of $G_y$-torsors 
    is smooth and affine over $[\Spec A/G_x]$.  The isomorphism $\alpha_0$ defines a section 
    of $\cI \to [\Spec A/G_x]$ over $BG_x$. By \cite[Prop. 7.18]{ahr2}, there exists an 
    \'etale neighborhood $\Spec R \to \Spec A^{G_x}$ such that $\alpha_0$ extends to a 
    section of $\cI \to [\Spec A/G_x]$ over 
    $\Spec R \times_{\Spec A^{G_x}} [\Spec A/G_x]$. 
    Therefore, after replacing $A$ with $R \tensor_{A^{G_x}} A$, there is a 
    2-isomorphism of $p$ and $q$.

    The addendums follow from \cite[Prop. 5.7]{ahr2}.
\end{proof}

There is a stronger version when $\cX \to \cY$ is a smooth morphism of smooth noetherian algebraic stacks.  For a point $x \in \cX(\bk)$, we let $N_x$ denote the normal space to $x$, viewed as a $G_x$-representation.  Explicitly, as $x \in |\cX|$ is locally closed, there is an open neighborhood $\cU \subset \cX$ of $x$ such that the inclusion $B G_x \to \cU$ is a closed immersion defined by an ideal sheaf $\cI$, and we take $N_x = (\cI/\cI^2)^{\vee}$.  Recall also that a morphism $\cX \to \cY$ of algebraic stacks admitting good moduli spaces $X$ and $Y$ is called \emph{strongly \'etale} if the induced map $X \to Y$ is \'etale and $\cX \cong X \times_Y \cY$.

\begin{corollary}
  Let $f \co \cX \to \cY$ be a smooth morphism of smooth quasi-separated algebraic stacks with affine stabilizers.  Let $x \in \cX(\bk)$ and $y=f(x) \in \cY(\bk)$.  Assume that the stabilizers $G_x$ and $G_y$ are 
  linearly reductive.  Then there exist
  \begin{enumerate}[(1)]
    \item affine schemes $\Spec A$ and $\Spec B$ with actions of $G_x$ and $G_y$, respectively;
    \item $\bk$-points $x' \in \Spec A$ and $y' \in \Spec B$; 
    \item a morphism $g \co \Spec A \to \Spec B$ of affine schemes equivariant with respect to $G_x \to G_y$; and 
    \item \'etale morphisms $([\Spec A / G_x], x') \to (\cX,x)$ and  $([\Spec B / G_y], y') \to (\cY,y)$ inducing isomorphisms of stabilizer groups at $x'$ and $y'$; and
    \item strongly \'etale morphisms $([\Spec A / G_x], x') \to ([N_x/G_x],0)$ and  $([\Spec B / G_y], y') \to ([N_y/G_y],0)$.
  \end{enumerate}
  such that the diagram
  $$\xymatrix{
    [N_x/G_x] \ar[d]
      & [\Spec A/G_x] \ar[r] \ar[d]^g \ar[l]
      & \cX \ar[d]^f \\
    [N_y/G_y]
      & [\Spec B/G_y] \ar[r]  \ar[l]
      & \cY
  }$$
  is commutative.
  Moreover, if $\cX$ and $\cY$ have separated diagonal (resp. affine diagonal), then the morphisms $[\Spec A/G_x] \to \cX$ and $[\Spec B/G_y] \to \cY$ can be arranged to be representable (resp. affine).
\end{corollary}

\begin{remark}
  The above theorem was independently proved in \cite[Thm.~2.1]{dCHN}.
\end{remark}

\begin{proof}
  The existence of the right commutative diagram follows from \Cref{thm:luna-etale-slice-theorem-for-morphisms}.  The existence of the strongly \'etale morphisms $([\Spec A / G_x], x') \to ([N_x/G_x],0)$ and  $([\Spec B / G_y], y') \to ([N_y/G_y],0)$ is a consequence of \cite[Thm.~1.2]{ahr}.  Under the natural map $[N_x/G_x] \to [N_y/G_y]$, the left diagram can be arranged to be commutative by the same argument used in the proof of \Cref{thm:luna-etale-slice-theorem-for-morphisms}.
\end{proof}

\subsection{A torus-equivariant Luna \'etale slice theorem}

We record a special case of \Cref{thm:luna-etale-slice-theorem-for-morphisms} when the target $\cY$ is the classifying stack of a torus $T$. 
Let $X$ be a quasi-separated and locally of finite type algebraic stack with a $T$-action, and set $\cX = [X/T]$ to be the quotient stack over $BT$. For $x \in X(\bk)$, we let $G_x = \Aut_{X(\bk)}(x)$, $\Gamma_x = \Aut_{\cX(\bk)}(x)$, and $T_x = \im(\Gamma_x \to T)$.   In other words, $G_x$ and $\Gamma_x$ sit in the following exact sequences:
\begin{equation*}
  1 \to G_x \to \Gamma_x \to T_x \to 1 \quad \text{and} \quad
  1 \to T_x \to T \to T/T_x \to 1. 
\end{equation*}
Observe that if $W$ is a scheme with a $\Gamma_x$-action, then $[W/\Gamma_x]$ is naturally an algebraic stack over $BT$ via $[W/\Gamma_x] \to B\Gamma_x \to BT$. The fiber product 
$[W/\Gamma_x] \times_{BT} \Spec \bk$ is isomorphic to $[(W \times T)/\Gamma_x]$, where $\Gamma_x$ acts diagonally, and inherits a $T$-action such that $[W/\Gamma_x] \cong [ [(W \times T)/\Gamma_x] / T]$.

\begin{corollary}
  \label{cor:equivariant-luna}
  Let $\cX = [X/T]$, where $X$ is a quasi-separated and locally of finite type algebraic stack over $\bk$ with affine stabilizers and with an action by a torus $T$. 
  Let $x \in \cX(\base)$ be a point with linearly reductive stabilizer (which is automatic if $X$ is a tame Deligne--Mumford stack). 
  Then there exists an affine scheme $W$ with a $\Gamma_x$-action and an \'etale morphism $([W/\Gamma_x],w) \to (\cX,x)$ over $BT$ inducing an isomorphism of stabilizer groups at $w$.  \epf
\end{corollary}

\begin{remark} The following are special cases of the above corollary:
  \begin{enumerate}[(1)]
      \item If the $\bG_m$-action on $X$ is trivial, i.e. $\cX = X \times B \bG_m$, then this follows from the Luna \'etale slice theorem for algebraic stacks \cite[Thm.~1.1]{ahr}.
      \item If $X$ is an algebraic space, then this is Sumihiro's theorem for torus actions generalized to algebraic spaces \cite[Thm. 4.1]{ahr}.   
  \end{enumerate}
\end{remark}

\subsection{Proof of the general case of \Cref{thm:main}}

\begin{proof}[Proof of \Cref{thm:main}]
  For every $x \in \cX$, we use \Cref{cor:equivariant-luna} to obtain an \'etale morphism
  \begin{equation} \label{eqn:etale-ngbd}
    \cU = [\Spec A/\Gamma_x] = [U/\bG_m^n] \to [X/\bG_m^n] = \cX 
  \end{equation}
  over $B \bG_m^n$ and a preimage $u \in \cU(\base)$ of $x$, where $U = [\Spec A/G_x]$.  By \Cref{prop:quotient-stack-case}, the theorem holds for $\cU$.  This gives a unique tuple $r = (r_1, \ldots, r_n)$ and a central closed subgroup $\mu_{r,\cU} \subset I_{\cU/B\bG_m^n}$.  Moreover, the construction of the subgroup $\mu_{r,\cU}$ is canonical: writing $\Gamma_x$ as an extension $1 \to G_x \to \Gamma_x \to \bG_m^n \to 1$, then $\mu_r = G_x \cap \Gamma_x^0 \subset \Gamma_x$ acts trivially on $\Spec A$ and this defines $\mu_{r,\cU}$.  As $\cX$ is connected, the tuple $r = (r_1, \ldots, r_n)$ is the same for every \'etale neighborhood.  Moreover, covering $\cX$ with \'etale neighborhoods as in \eqref{eqn:etale-ngbd}, we have canonical isomorphisms of the subgroups of the inertia over intersections, and this provides descent data for the construction of a central closed subgroup $\mu_{r,\cX} \subset I_{\cX/B\bG_m^n}$.  
  
  The rigidification $\cX \rigidify \mu_{r,\cX}$ is the quotient stack of $X \rigidify \mu_{r,X}$ by $\bG_m^n$.  An \'etale neighborhood \eqref{eqn:etale-ngbd} induces \'etale neighborhoods $\cU \rigidify \mu_{r,\cU}  \to \cX \rigidify \mu_{r,\cX}$ and $U \rigidify \mu_{r,U}  \to X \rigidify \mu_{r,X}$.  Since $\bG_m^n$ acts trivially on $U \rigidify \mu_{r,U}$, it acts trivially on $Y:=X \rigidify \mu_{r,X}$.  This gives a factorization 
  $$
    \cX \to \cX \rigidify \mu_{r,\cX}  \cong Y \times B\bG_m^n \to X_{\cms} \times B\bG_m^n, 
  $$
  which satisfies the universal property by \Cref{prop:universal-properties}, yielding part \eqref{thm:main1}.  The factorizations in parts \eqref{thm:main2} and \eqref{thm:main3} also follow from \'etale descent.
\end{proof}

\bibliography{refs}
\bibliographystyle{amsalpha}
\end{document}